\documentclass[11pt,reqno]{amsart}  
\usepackage{amsfonts,amssymb,amsmath}
\usepackage{amsbsy}
\usepackage{latexsym} 
\usepackage{mathtools,mathdots} 

\numberwithin{equation}{section} 
\renewenvironment{proof}{{\bfseries Proof.}}{\qed} 
\makeatletter
\def\subsection{\@startsection{subsection}{3}%
  \z@{.5\linespacing\@plus.7\linespacing}{.1\linespacing}%
  {\bfseries}}
\newtheorem{theorem}{Theorem}[section] 
 
\newtheorem{corollary}[theorem]{Corollary}

\newtheorem{remark}[theorem]{Remark} 
\usepackage[colorlinks=true,citecolor=cyan,urlcolor=blue,linkcolor=blue]{hyperref}

\def\R{\mathbb {R}}
\def\C{\mathbb {C}}

\def\Z{\mathbb {Z}}

\def\ZC{\mathcal {Z}}

\def\OC{\mathcal {O}}

\def\g{\mathfrak {g}}
\def\h{\mathfrak {h}}
\def\p{\mathfrak {p}}
\def\k{\mathfrak {k}}
\def\s{\mathfrak {s}}
\def\a{\mathfrak {a}}

\def\o{\mathfrak {o}}
\def\l{\mathfrak {l}}
\def\z{\mathfrak {z}}
\def\m{\mathfrak {m}}

\begin{document} 
 
\title [On the nilpotent orbits in Lie algebras]{On the second cohomology of
nilpotent orbits in exceptional Lie algebras}  
 
\author[P. Chatterjee]{Pralay Chatterjee}

\address{The Institute of Mathematical Sciences, HBNI, C.I.T. Campus, 
  Tharamani, Chennai 600113, India}  
 \email{pralay@imsc.res.in} 

\author[C. Maity]{Chandan Maity}
 
\address{The Institute of Mathematical Sciences, HBNI, C.I.T. Campus, 
  Tharamani, Chennai 600113, India}  
 \email{cmaity@imsc.res.in} 



\begin{abstract}
In \cite{BC}, the second de Rham cohomology groups of nilpotent orbits in all the 
complex simple Lie algebras are described. 
In this paper we consider non-compact non-complex exceptional Lie algebras, and compute
the dimensions of the second cohomology groups for most of the nilpotent orbits. 
For the rest of cases of nilpotent orbits, which are not covered in the above computations, we obtain upper bounds 
for the dimensions of the second cohomology groups.  
\end{abstract}

\maketitle

\section{Introduction}
Let $G$ be a connected real simple Lie group with Lie algebra $\g$. An element $X \in 
\g$ is called {\it nilpotent} if ${\rm ad}(X): \g \rightarrow \g$ is a nilpotent 
operator. Let $\OC_X := \{{\rm Ad}(g)X \mid g \in G \}$ be the corresponding {\it nilpotent orbit}  
under the adjoint action of $G$ on $\g$.
Such nilpotent orbits form a rich class of 
homogeneous spaces, and they are studied at the interface of several disciplines in 
mathematics such as Lie theory, symplectic geometry, representation theory, algebraic 
geometry. Various topological aspects of such orbits have drawn attention over the 
years; see \cite{CM}, \cite{M} and references therein for an account. In 
\cite[Proposition 1.2]{BC} for a large class of semisimple Lie groups a criterion is 
given for the exactness of the Kostant-Kirillov two form on arbitrary adjoint orbits 
which in turn led the authors asking the natural question of describing the full 
second cohomology groups of such orbits. Towards this, in \cite{BC}, the second 
cohomology groups of nilpotent orbits in all the complex simple Lie algebras, under 
the adjoint actions of the corresponding complex groups, are computed. In this paper
we continue the program of studying the second cohomology groups of nilpotent orbits which was initiated in  
\cite{BC}. We compute the second cohomology groups for
most of the nilpotent orbits in non-compact non-complex exceptional
Lie algebras, and for the rest of the nilpotent orbits in non-compact non-complex exceptional
Lie algebras we give upper bounds of the dimensions of 
second cohomology groups; see Theorems \ref{G-2(2)}, \ref{F-4(4)}, \ref{F-4(-20)}, \ref{E-6(6)}, \ref{E-6(2)}, \ref{E-6(-14)}, \ref{E-6(-26)},  \ref{E-7(7)}, \ref{E-7(-5)}, \ref{E-7(-25)}, \ref{E-8(8)}, \ref{E-8(-24)}.
In particular, our computations yield that the second cohomologies vanish for all the nilpotent orbits in $F_{4(-20)}$ and $E_{6(-26)}$.

\section{Notation and background}
In this section we fix some general notation, and mention a basic result which will be used in this paper. 
A few specialized notation are defined as and when they occur later.

The {\it center} of a Lie algebra $\g$ is denoted by $ \z (\g)$. We 
denote Lie groups by capital letters, and unless mentioned otherwise we denote their Lie algebras by the corresponding lower case German letters. Sometimes, for convenience,  the Lie algebra of a Lie group $G$ is also denoted by ${\rm Lie} (G)$.
The connected component of a Lie group $G$ containing the identity element is denoted by $G^{\circ}$.
For a subgroup $H$ of $G$ and a subset  $S$ of $\g$, the
subgroup of $H$ that fixes $S$ point wise is called the {\it centralizer} of $S$ in $H$ and is denoted by $\ZC_{H} (S)$. Similarly, for a Lie subalgebra $\h 
\subset \g$ and a subset $S \subset \g$, by  $ \z_\h (S)$ we will denote the subalgebra consisting elements of $\h$ that commute with every element of $S$. 

If $G$ is a Lie group with Lie algebra $\g$, then it is immediate that the coadjoint action of $G^{\circ}$ on 
$\z(\k)^*$ is trivial; in particular, one obtains a natural action of $G/G^{\circ}$ on $\z (\k)^*$. We denote by
$[\z (\g)^*]^{G/G^{\circ}}$ the space of fixed points of $\z (\g)^*$ under the action of $G/G^{\circ}$. 

For a real semisimple Lie group $G$, an element $X \in \g$ is called {\it nilpotent} if ${\rm ad}(X): \g \rightarrow \g$ is a nilpotent operator.
A {\it nilpotent orbit} is an orbit of a nilpotent element in $\g$ under the adjoint representation of $G$; for a nilpotent element
$X \in \g$ the corresponding nilpotent orbit ${\rm Ad}(G) X$ is denoted by $\OC_X$.

For a $\g$ be a Lie algebra over $\R$ a subset $\{X,H,Y\} \subset \g$ is said to be {\it $\s\l_2(\R)$-triple} 
if $X \neq 0$, $[H, X] = 2X, [H,Y] = -2Y$ and $[X, Y] =H$. It is immediate that, if
$\{X,H,Y\} \subset \g$ is a  $\s\l_2(\R)$-triple then $\text{Span}_\R \{X,H,Y\}$ is a $\R$-subalgebra of $\g$ which is isomorphic to the Lie algebra $\s\l_2(\R)$.
We now recall the well-known Jacobson-Morozov theorem (see \cite[Theorem 9.2.1]{CM}) which ensures that if 
$X\in \g$ is a non-zero nilpotent element in a real semisimple Lie algebra  $\g$, then there exist $H,Y$ in $\g$ such that 
$\{X,H,Y\} \subset \g$ is a  $\s\l_2(\R)$-triple.

To facilitate the computations in \S \ref{exceptional} we need the following result.

\begin{theorem}\label{thm-nilpotent-orbit}
Let $G$ be an algebraic group defined over $\R$ which is $\R$-simple.
Let $X \in {\rm Lie}\, G(\R)$, $X \neq 0$ be a nilpotent element, and $\OC_X$ be the orbit of $X$ under the adjoint action
of the identity component $G (\R)^\circ$ on ${\rm Lie}\, G(\R)$. Let $\{X,H,Y\}$ be a $\s\l_2(\R)$-triple in $ {\rm Lie}\, G(\R)$. Let $K$ be a maximal compact subgroup in $\ZC_{G(\R)^ \circ}(X,H,Y)$, and  $M$ be a maximal compact subgroup in $G(\R)^\circ$ containing $K$. 
Then
  $$ 
  { H}^2(\OC_X, \R)  \simeq [(\z(\k) \cap [\m,\m])^*]^{K/K^{\circ}}.
  $$
In particular, $\dim_\R  H^2(\OC_X, \R) \leq \dim_\R \z(\k)$.
\end{theorem}

The above theorem follows from \cite[Lemma 3.7.3]{CM} and 
a description of the second cohomology groups of homogeneous spaces which generalizes \cite[Theorem 3.3]{BC}. The details of the proof of Theorem
\ref{thm-nilpotent-orbit} and the generalization of \cite[Theorem 3.3]{BC}, mentioned above, will appear elsewhere.

\section{The second cohomology groups of nilpotent orbits}
\label{exceptional}

In this section we study the second cohomology of the nilpotent orbits in non-compact non-complex exceptional
Lie algebras over $\R$. The results in this section depend on the results of \cite[Tables VI-XV]{D1}, \cite[Tables VII-VIII]{D2} and \cite[Tables 1-12]{K}.
We refer to \cite[Chapter 9]{CM}, \cite{D1} and \cite{D2} for the generalities required in this section.
We begin by recalling the  parametrization of nilpotent orbits in this set-up.

\subsection{Parametrization of nilpotent orbits in exceptional Lie algebras}\label{parametrization-exceptional}
We follow the parametrization of nilpotent orbits in non-compact non-complex 
exceptional Lie algebras as given in \cite[Tables VI-XV]{D1} and \cite[Tables VII-VIII]{D2}.
We consider the nilpotent orbits in $\g$ under the action of ${\rm Int}\,\g$, where
$\g$ is a non-compact non-complex real exceptional Lie algebra.
We fix a semisimple algebraic group $G$ defined over $\R$ such that $\g = {\rm Lie} (G(\R))$.
Here $G(\R)$ denotes the associated real semisimple Lie group of the $\R$-points of $G$. Let $G (\C)$ be the associated complex semisimple Lie group consisting of the $\C$-points of $G$. It is easy to see that 
orbits in $\g$ under the action of ${\rm Int}\,\g$ are the same as the orbits in $\g$ under the action of $G(\R)^\circ$.
Thus in this set-up, for a nilpotent element $X \in \g$, we set $ \OC_X := \{ {\rm Ad}(g) X \mid g \in G(\R)^\circ \}$. Let $\g = \m + \p$ be a Cartan decomposition and $\theta$ be the corresponding Cartan involution. Let $\g_\C$ be the Lie algebra of $G(\C)$. Then $\g_\C$ can be identified with the complexification of $\g$. Let $\m_\C, \p_\C$ be the $\C$-spans
of $\m$ and $\p$ in $\g_\C$, respectively. Then $\g_\C= \m_\C + \p_\C$. Let $M_\C$ be the connected subgroup of $G (\C)$ with Lie algebra $\m_\C$. 
Recall that, if $\g$ is as above and $\g$ is different from both $E_{6(-26)}$ and $E_{6(6)}$, then $\g$ is of inner type, or equivalently, ${\rm rank}\, \m_\C = {\rm rank }\, \g_\C$.
When $\g$ is of inner type, the nilpotent orbits are parametrized by a finite sequence of integers of length $l$ where $l:= {\rm rank}\, \m_\C = {\rm rank }\, \g_\C$.
When $\g$ is not of inner type, that is, when $\g$ is either $E_{6(-26)}$ or $E_{6(6)}$, then the nilpotent orbits are parametrized by a finite sequence of integers of length $4$.

Let $X'\in \g$ be a nonzero nilpotent element, and $\{X', H', Y' \} \subset \g$ be a $\s\l_2(\R)$-triple. Then $\{X', H', Y' \}$ is $G(\R)$-conjugate to another $\s\l_2(\R)$-triple $\{ \widetilde X,\widetilde H,\widetilde Y \}$ in $\g$ such that $\theta (\widetilde H) = - \widetilde H$, $\theta (\widetilde X) = - \widetilde Y$. Set $E := (\widetilde H - i (\widetilde X + \widetilde Y))/2$, $F := (\widetilde H + i (\widetilde X + \widetilde Y))/2$ and $H := i (\widetilde X - \widetilde Y)$. 
Then $\{E,H,F \}$ is a $\s\l_2(\R)$-triple and 
$E,F \in \p_\C$ and $H \in \m_\C$. The $\s\l_2(\R)$-triple $\{E,H,F \}$ is then called a
$\p_\C$-{\it Cayley triple associated to $X'$}.

\subsubsection{Parametrization in exceptional Lie algebras of inner type}\label{parametrization-inner-type}
We now recall from  \cite[Column 2, Tables VI-XV]{D1}
the parametrization of non-zero nilpotent orbits in $\g$  when $\g$ is an exceptional Lie algebra of inner type.
Let $\h_\C \subset \m_\C$ be a Cartan subalgebra of $\m_\C$ such that $\h_\C \cap \m$ is a Cartan subalgebra of $\m$. As $\g$ is of inner type, $\h_\C$ is a Cartan subalgebra of $\g_\C$. 
Set $\h: = \h_\C \cap i\m$. Let $R, R_0$ be the root systems of $(\g_\C, \h_\C), (\m_\C, \h_\C)$, respectively. Let $B:= \{\alpha_1, \dotsc, \alpha_l \}$ be a basis of $R$. 
Let $B_e:= B \cup \{\alpha_0 \}$ where $\alpha_0$ is the negative of the highest root of $(R, B)$.
Then there exists an unique basis of $R_0$, say $B_0$, such that $B_0 \subset B_e$.
Let $C_0$ be the closed Weyl chamber of $R_0$ in $\h$ corresponding to the basis $B_0$.
Let $l_0$ be the rank of $[\m_\C, \m_\C]$. Then either $l_0 = l$ or $l_0=l-1$.  If $l_0= l$ we set $B'_0 := B_0$. If $l_0=l-1$ (in this case we have $B_0 \subset B$) we set $B'_0 := B$. Clearly, $\# B'_0 = l$. We enumerate $B'_0 :=  \{\beta_1, \dotsc , \beta_l \}$ as in \cite[7, p. 506 and Table IV]{D1}.
Let $X\in \g$ be a nonzero nilpotent element, and $\{E,H,F\}$ be a $\p_\C$-Cayley triple (in $\g_\C$) associated to $X$.
Then  ${\rm Ad}(M_\C ) H \cap C_0$ is a singleton set, say $\{ H_0\}$. The element $H_0$ is called {\it the characteristic} of the orbit ${\rm Ad}(M_\C ) E$ as it determines the orbit $M_\C \cdot E$ uniquely. Consider the map from the set of nilpotent orbits in $\g$ to the set of integer sequences of length $l$, which assigns the sequence $\beta_1(H_0), \dotsc , \beta_l(H_0)$ to each nilpotent orbits $\OC_X$.
In view of the Kostant-Sekiguchi theorem (cf. \cite[Theorem 9.5.1]{CM}), this gives a bijection between the set of nilpotent orbits in $\g$ and 
the set of finite sequences of the form $\beta_1(H_0), \dotsc , \beta_l(H_0)$ as above. We use this 
parametrization while dealing with nilpotent orbits in exceptional Lie algebras of inner type.

\subsubsection{\texorpdfstring{Parametrization in $E_{6(-26)}$ or $E_{6(6)}$}{Lg}}\label{parametrization-outer-type}

We now recall from \cite[Column 1, Tables VII-VIII]{D2} 
the parametrization of non-zero nilpotent orbits in $\g$  when $\g$ is either  
$E_{6(-26)}$ or $E_{6(6)}$. We need a piece of notation here : henceforth, for a Lie algebra $\a$ over $\C$ and an automorphism $\sigma \in {\rm Aut}_\C \, \a$, the Lie
subalgebra consisting of the fixed points of $\sigma$ in $\a$, is denoted by $\a^\sigma$.     
Let now $\h_\C$ be a Cartan subalgebra of $\g_\C$ (we point out the difference of our notation with that in \cite{D2}; 
$\g$ and $\h$ of \cite[\S 1]{D2} are denoted here by $\g_\C$ and $\h_\C$, respectively).

Let $\g = E_{6(-26)}$.
Let $\tau$ be the involution of $\g_\C$ as defined in \cite[p. 198 ]{D2}  which keeps $\h_\C$ invariant.
Then the subalgebra $\g^{\tau}_\C$ is of type $F_4$, and $\h^{\tau}_\C$ is a Cartan subalgebra of $\g^{\tau}_\C$.
Let $G(\C)^{\tau}$ be the connected Lie subgroup of $G(\C)$ with Lie algebra $\g^\tau_\C$. Let $\{\beta_1,\beta_2,\beta_3, \beta_4\}$ be the  
simple roots of $(\g^{\tau}_\C,  \h^{\tau}_\C)$ as defined in \cite[(1), p. 198]{D2}.
Let $X\in E_{6(-26)}$ be a nonzero nilpotent element. Let $\{E,H,F\}$ be a $\p_\C$-Cayley triple (in $\g_\C$) associated to $X$. 
Then $H \in \g^{\tau}_\C$ and $E,F \in \g^{-\tau}_\C$. We may further assume that $H \in \h^{\tau}_\C$. Then the finite sequence of integers $\beta_1(H), \beta_2(H),\beta_3(H), \beta_4(H)$ determine the orbit ${\rm Ad}(G(\C)^{\tau}) \, E$ uniquely; see \cite[p. 204]{D2}.

Let $\g = E_{6(6)}$.
Let $\tau'$  be the involution of $\g_\C$ as defined in \cite[p. 199 ]{D2} which keeps $\h_\C$ invariant.
Then the subalgebra $\g^{\tau'}_\C$ is of type $C_4$, and $\h^{\tau'}_\C$ is a Cartan subalgebra of $\g^{\tau'}_\C$.
Let $G(\C)^{\tau'}$ be the connected Lie subgroup of $G(\C)$ with Lie algebra $\g^{\tau'}_\C$.
Let $\{\beta_0, \beta_1,\beta_2,\beta_3\}$ be the  
simple roots of $(\g^{\tau'}_\C, \h^{\tau'}_\C)$ as defined in \cite[p. 199]{D2}.
Let $X\in E_{6(6)}$ be a nonzero nilpotent element. Let $\{E,H,F\}$ be a $\p_\C$-Cayley triple (in $\g_\C$) associated to $X$.
Then  $H \in \g^{\tau'}_\C$ and $E,F \in \g^{-\tau'}_\C$. We may further assume that $H \in \h^{\tau'}_\C$. It then follows that the finite sequence of integers $\beta_0(H), \beta_1(H), \beta_2(H), \beta_3(H)$ determine  the orbit ${\rm Ad} (G(\C)^{\tau'}) \, E$ uniquely; see \cite[p. 204]{D2}.

\subsection{Nilpotent orbits of three types}
For the sake of convenience of writing the proofs that appear in the later part \S \ref{exceptional}, it will be useful to divide the nilpotent orbits in the following three types. 
Let $X\in \g$ be a nonzero nilpotent element, and $\{X,H,Y\}$ be a $\s\l_2(\R)$-triple in $\g$. 
Let $G$ be as in the beginning of \S \ref{parametrization-exceptional}. 
Let $K$ be a maximal compact subgroup in $\ZC_{G(\R)^\circ}(X,H,Y)$, and $M$ be a maximal compact subgroup in $G(\R)^\circ$ containing $K$. 
A nonzero nilpotent orbit $\OC_X$ in $\g$ is said to be of 
\begin{enumerate}
 \item  {\it type I} if $\z(\k) \neq 0$, $K/K^\circ = {\rm Id}$ and  $\m= [\m,\m]$;
 
 \item  {\it type II } if either $\z(\k) \neq 0$, $K/K^\circ \neq {\rm Id}$, $\m= [\m,\m]$; or $\z(\k) \neq 0$, $\m \neq [\m,\m]$;
 
 \item  {\it type III } if $\z(\k) = 0$.
\end{enumerate}

 In what follows we will use the next result repeatedly.

\begin{corollary}\label{cor-dim-types-ABC}
 Let $\g$ be a real simple non-compact exceptional Lie algebra. Let $X \in \g$ be a nonzero nilpotent element.
 \begin{enumerate}
  \item  If the orbit $\OC_X$ is of type I, then $\dim_\R  H^2(\OC_X, \R) = \dim_\R \z(\k)$.
  
  \item  If the orbit $\OC_X$ is of type II, then $\dim_\R  H^2(\OC_X, \R) \leq \dim_\R \z(\k)$.
  
  \item  If the orbit $\OC_X$ is of type III, then  $\dim_\R  H^2(\OC_X, \R) =0$.
 \end{enumerate}
\end{corollary}

\begin{proof}
 The proof of the corollary follows immediately from Theorem \ref{thm-nilpotent-orbit}.
\end{proof}

Let $\g$ be as above.
In the proofs of our results in the following subsections we use the description of 
a Levi factor of $\z_{\g} (X)$ for each nilpotent element $X$ in $\g$,
as given in the last columns of \cite[Tables VI-XV]{D1} and \cite[Tables VII-VIII]{D2}. This enables us  
compute the dimensions $\dim_\R \z(\k)$ easily.  We also use \cite[Column 4, Tables 1-12]{K} for 
the component groups for each nilpotent orbits in $\g$.

\subsection{\texorpdfstring{Nilpotent orbits in the non-compact real form of $G_2$}{Lg}}
Recall that up to conjugation there is only one non-compact real form of $G_2$. We denote it by $G_{2(2)}$. 
There are only five nonzero nilpotent orbits in $G_{2(2)}$; see \cite[Table VI, p. 510]{D1}. Note that in this case we have $\m= [\m,\m]$.
\begin{theorem}\label{G-2(2)}
Let the parametrization of the nilpotent orbits be as in \S \ref{parametrization-inner-type}. 
 Let $X$ be a nonzero nilpotent element in $G_{2(2)}$. 
\begin{enumerate}
  \item If the parametrization  of the orbit $\OC_X$ is given by either   
   $1~1$ or $~1~3$, then $\dim_\R H^2(\OC_X, \R)$ $= 1$.
   
  \item If the parametrization of the orbit $\OC_X$ is given by any of 
  $ ~2~2,~ 0~4, ~ 4~8$, then $\dim_\R H^2(\OC_X, \R)$ $= 0$.
\end{enumerate}
\end{theorem}

\begin{proof}
 From \cite[Column 7, Table VI, p. 510]{D1} we have $\dim_\R \z(\k) =1$ and from \cite[Column 4, Table 1, p. 247]{K} 
 we have $K/K^\circ= {\rm Id}$ for the nilpotent orbits as in (1). Thus these are of type I. 
 We refer to \cite[Column 7, Table VI, p. 510]{D1} for the orbits as given in (2). These orbits are of type III as $\dim_\R \z(\k) = 0$.  
 In view of the Corollary \ref{cor-dim-types-ABC} the conclusions follow.
\end{proof}

\subsection{\texorpdfstring{Nilpotent orbits in non-compact real forms of $F_4$}{Lg}}

Recall that up to conjugation there are two non-compact real forms of $F_4$. They are denoted by $F_{4(4)}$ and $F_{4(-20)}$.

\subsubsection{\texorpdfstring {Nilpotent orbits in $F_{4(4)}$.}{Lg}}
There are 26 nonzero nilpotent orbits in $F_{4(4)}$; see \cite[Table VII, p. 510]{D1}. Note that in this case we have $\m=[\m,\m]$.
\begin{theorem}\label{F-4(4)}
Let the parametrization of the nilpotent orbits be as in \S \ref{parametrization-inner-type}.
 Let $X$ be a nonzero nilpotent element in $F_{4(4)}$.
 \begin{enumerate}
  \item Assume the parametrization of the orbit $\OC_X$ is given by any of the sequences :\\
  $001~1,~ 001~3,~ 110~2,~ 111~1,~ 131~3$.
  Then  $\dim_\R H^2(\OC_X, \R)= 1$. \vspace{.1cm}
  
  \item  Assume the parametrization of the orbit $\OC_X$ is given by any of the sequences :\\
   $100~2,~ 200~0,~ 103~1,~ 111~3,~ 204~4$. Then  
  $\dim_\R H^2(\OC_X, \R)\leq  1$.  \vspace{.1cm}
  
  \item  If the parametrization  of the orbit $\OC_X$ is either $101~1$ or $012~2$, 
         then  $\dim_\R H^2(\OC_X, \R)\leq  2$.  \vspace{.1cm}
  
  \item  If $\OC_X$ is not given by the parametrizations as in (1), (2), (3) above (\# of such orbits are 14),
   then we have $\dim_\R H^2(\OC_X, \R)= 0$.
\end{enumerate}
\end{theorem}

\begin{proof} 
For the Lie algebra $F_{4(4)}$, we can easily compute $\dim_\R \z(\k)$ from the last column of \cite[Table VII, p. 510]{D1} 
and $K/K^\circ$ from \cite[ Column 4, Table 2, pp. 247-248]{K}.

For the orbits $\OC_X$, as in (1), we have $\dim_\R \z(\k) = 1$ and $K/K^\circ = {\rm Id}$. Hence these are of type I.
For the orbits $\OC_X$, as in (2), we have $\dim_\R \z(\k) = 1$ and  $K/K^\circ \neq {\rm Id}$; hence they are of type II. 
For the orbits $\OC_X$, as in (3), we have $\dim_\R \z(\k) = 2$ and  $K/K^\circ \neq {\rm Id}$. Hence these are also of type II. 
The rest of the 14 orbits, which are not given by the parametrizations in (1), (2), (3), are of type III as $\z(\k) = 0$. 
Now the theorem follows from Corollary \ref{cor-dim-types-ABC}.  
\end{proof}

\subsubsection{\texorpdfstring {Nilpotent orbits in $F_{4(-20)}$.}{Lg}}
There are two nonzero nilpotent orbits in $F_{4(-20)}$; see \cite[Table VIII, p. 511]{D1}. 
\begin{theorem}\label{F-4(-20)} 
 For all the nilpotent elements $X$ in $F_{4(-20)}$ we have $\dim_\R H^2(\mathcal{O}_X,\R) = 0$.
\end{theorem}

\begin{proof}
As the theorem follows trivially when $X =0$ we assume that $X \neq 0$. We follow the parametrization of nilpotent orbits as in \S \ref{parametrization-inner-type}.
From the last column of \cite[Table VIII, p. 511]{D1} we conclude that $\z(\k) = 0$. Hence the nonzero nilpotent orbits are of type III.
Using Corollary \ref{cor-dim-types-ABC} (3) we have $\dim_\R H^2(\mathcal{O}_X,\R) = 0$.
\end{proof}

\subsection{\texorpdfstring{Nilpotent orbits in non-compact real forms of $E_6$}{Lg}}

Recall that up to conjugation there are four non-compact real forms of $E_6$. They are denoted by  $E_{6(6)}$, $E_{6(2)}$, $E_{6(-14)}$ and $E_{6(-26)}$.

\subsubsection{\texorpdfstring {Nilpotent orbits in $E_{6(6)}$.}{Lg}}
There are 23 nonzero nilpotent orbits in $E_{6(6)}$; see \cite[Table VIII, p. 205]{D2}. Note that in this case we have $\m= [\m,\m]$.

\begin{theorem}\label{E-6(6)} 
Let the parametrization of the nilpotent orbits be as in \S \ref{parametrization-outer-type}.
Let $X$ be a nonzero nilpotent element in $E_{6(6)}$.
 \begin{enumerate}
 \item If the parametrization of the orbit $\OC_X$ is given by either    
  $1001$ or $1101$ or $1211$,
  then $\dim_\R H^2(\OC_X, \R) = 1$.
  
 \item 
Assume the parametrization of the orbit $\OC_X$ is given by any of the sequences :\\ 
  $0102, 0202, 1010$, $2002, 1011$.
  Then $\dim_\R H^2(\OC_X, \R) \leq 1$.
  
  \item  If $\OC_X$ is not given by the parametrizations as in (1), (2) above (\# of such orbits are 15),
   then we have $\dim_\R H^2(\OC_X, \R)= 0$.
 \end{enumerate}
\end{theorem}

\begin{proof}
For the Lie algebra $E_{6(6)}$, we can easily compute $\dim_\R \z(\k)$ from the last column of \cite[Table VIII, p. 205]{D2} 
and $K/K^\circ$ from \cite[Column 4, Table 4, p.253]{K}. 
As pointed out in the $1^{{\rm st}}$ paragraph of \cite[p. 254]{K}, there is an error in row 5 of \cite[Table VIII, p. 205]{D2}. 
Thus when $\OC_X$ is given by the parametrization 2000 it follows from  \cite[p. 254]{K} that $\z(\k)=0$.

We have $\dim_\R \z(\k) =1$ and $K/K^\circ = {\rm Id}$ for the orbits given in (1). Thus these orbits are of type I.
For the orbits, as in (2), we have $\dim_\R \z(\k) = 1$ and  $K/K^\circ= \Z_2$. Hence, the orbits in (2) are of type II.  
For rest of the 15 nonzero nilpotent orbits, which are not given by the parametrizations of (1), (2), are of type III as $\dim_\R \z(\k) = 0$.
Now the results follow from Corollary \ref{cor-dim-types-ABC}. 
\end{proof}

\subsubsection{\texorpdfstring {Nilpotent orbits in $E_{6(2)}$.}{Lg}}
There are 37 nonzero nilpotent orbits in $E_{6(2)}$; see \cite[Table IX, p. 511]{D1}. Note that in this case we have $\m= [\m,\m]$.

\begin{theorem}\label{E-6(2)}
  Let the parametrization of the nilpotent orbits be as in \S \ref{parametrization-inner-type}. 
  Let $X$ be a nonzero nilpotent element in $E_{6(2)}$.
\begin{enumerate}
  \item 
  Assume the parametrization of the orbit $\OC_X$ is given by any of the sequences :\\
  $00000~4,~ 00200~2, ~02020~0, ~00400~8,~ 22222~2, ~04040~4,~ 44044~4,~ 44444~8$.\\
   Then $\dim_\R H^2(\OC_X, \R)= 0$. \vspace{.2cm}
   
  \item Assume the parametrization of the orbit $\OC_X$ is given by any of the sequences :\\
   $10001~2,~ 10101~1,~ 21001~1,~ 10012~1,~ 11011~2,~ 01210~2,~ 10301~1,~ 11111~3, ~ 22022~0$.\\
   Then $\dim_\R H^2(\OC_X, \R)= 2$.\vspace{.2cm}
 
 \item If the parametrization of the orbit $\OC_X$ is given by either  
 $20002~0 $ or $ 00400~0$ or $ 02020~4$, 
  then $\dim_\R H^2(\OC_X, \R) \leq 2$. \vspace{.2cm}
 
 \item If the parametrization of the orbit $\OC_X$ is given by $20202~2$, 
 then $\dim_\R H^2(\OC_X, \R) \leq 1$. \vspace{.2cm}
 
 \item  If $\OC_X$ is not given by the parametrizations as in (1), (2), (3), (4) above (\# of such orbits are 16),
        then we have $\dim_\R H^2(\OC_X, \R)= 1$.
\end{enumerate}
\end{theorem}

\begin{proof} 
For the Lie algebra $E_{6(2)}$, we can easily compute $\dim_\R \z(\k)$ from the last column of \cite[Table IX, p. 511]{D1}
and $K/K^\circ$ from \cite[Column 4, Table 5, pp. 255-256]{K}.

We have $\z(\k) =0$ for the orbits, as given in (1), and these orbits are of type III. 
For the orbits, as given in (2), we have $\dim_\R \z(\k)= 2$ and $K/K^\circ = {\rm Id}$. Thus the orbits in (2) are of type I.   
For the orbits, as given in (3), we have $\dim_\R \z(\k)= 2$ and $K/K^\circ \neq {\rm Id}$, hence are of type II.  
For the orbits, as given in (4), we have $\dim_\R \z(\k)= 1$ and $K/K^\circ = \Z_2$. Thus this orbit is of type II.  
For the rest of 16 orbits, which are not given in any of (1), (2), (3), (4), we have $\dim_\R \z(\k)= 1$ and $K/K^\circ = {\rm Id}$. 
Thus these orbits are of type I. Now the conclusions follow from Corollary \ref{cor-dim-types-ABC}.
\end{proof}

\subsubsection{\texorpdfstring {Nilpotent orbits in $E_{6(-14)}$.}{Lg}}
There are 12 nonzero nilpotent orbits in $E_{6(-14)}$; see \cite[Table X, p. 512]{D1}. 
Note that in this case $\m \simeq \s\o_{10} \oplus \R$, and hence $[\m,\m] \neq \m $.

\begin{theorem}\label{E-6(-14)}
Let the parametrization of the nilpotent orbits be as in \S \ref{parametrization-inner-type}. Let $X$ be a nonzero nilpotent element in $E_{6(-14)}$.
\begin{enumerate}
  \item  If the parametrization of the orbit $\OC_X$ is given by $~40000-2$, then  $\dim_\R H^2(\OC_X, \R)= 0$. \vspace{.1cm}
  
  \item 
  If $\OC_X$ is not given by the above parametrization (\# of such orbits are 11),
   then we have  $\dim_\R H^2(\OC_X, \R) \leq 1$.
\end{enumerate}
\end{theorem}

\begin{proof}
For the Lie algebra $E_{6(-14)}$, we can easily compute $\dim_\R \z(\k)$ from the last column of \cite[Table X, p. 512]{D1}.
The orbit in (1) is of type III as $\z(\k) =0$, and hence  $\dim_\R H^2(\OC_X, \R)= 0$. 
The other 11 orbits are of type II as $\dim_\R \z(\k) = 1$ and $\m \neq [\m,\m]$. Hence $\dim_\R H^2(\OC_X, \R) \leq 1$.
\end{proof}

\subsubsection{\texorpdfstring { Nilpotent orbits in $E_{6(-26)}$.}{Lg}} 
There are two nonzero nilpotent orbits in $E_{6(-26)}$; see \cite[Table VII, p. 204]{D2}.

\begin{theorem}\label{E-6(-26)}
 For all the nilpotent element $X$ in $E_{6(-26)}$ we have $\dim_\R H^2(\mathcal{O}_X,\R) = 0$.
\end{theorem}

\begin{proof}
 As the theorem follows trivially when $X =0$ we assume that $X \neq 0$. We follow the parametrization of the nilpotent orbits as given in \S \ref{parametrization-outer-type}.
 The two nonzero nilpotent orbits in $E_{6(-26)}$ are of type III as $\z(\k) =0$; see last column of \cite[Table VII, p. 204]{D2}.
 Hence, by Corollary \ref{cor-dim-types-ABC} (3) we conclude that  $\dim_\R H^2(\mathcal{O}_X,\R) = 0$.
\end{proof}

\subsection{\texorpdfstring{Nilpotent orbits in non-compact real forms of $E_7$}{Lg}}

Recall that up to conjugation there are three non-compact real forms of $E_7$. They are denoted by $E_{7(7)}$, $E_{7(-5)}$ and $E_{7(-25)}$. 

\subsubsection{\texorpdfstring {Nilpotent orbits in $E_{7(7)}$.}{Lg}}
There are 94 nonzero nilpotent orbits in $E_{7(7)}$; see \cite[Table XI, pp. 513-514]{D1}. Note that in this case we have $\m = [\m,\m]$.

\begin{theorem}\label{E-7(7)}
Let the parametrization of the nilpotent orbits be as in \S \ref{parametrization-inner-type}. Let $X$ be a nonzero nilpotent element in $E_{7(7)}$.
 \begin{enumerate}
  \item If the parametrization of the orbit $\OC_X$ is given by $1011101$, then $\dim_\R H^2(\OC_X, \R)= 3$. \vspace{.2cm}
  
  \item Assume the parametrization of the orbit $\OC_X$ is given by any of the sequences :\\
  $1001001,~~ 1101011$,~~ $1111010$,~~ $ 0101111,\,~ 2200022, ~\,3101021,\,~ 1201013,~ \,1211121 , ~ \, 2204022$. \\
  Then $\dim_\R H^2(\OC_X, \R)= 2$.  \vspace{.2cm}
  
  \item Assume the parametrization of the orbit $\OC_X$ is given by any of the sequences :\\
  $0100010, ~ 1100100,~  0010011,~  3000100,~  0010003, ~  0102010,~ 0200020,~  2004002,~  2103101$, 
  $1013012, ~ 2020202,  ~1311111, ~ 1111131, ~ 1310301, ~  1030131, ~2220222,~  3013131, ~ 1313103$, 
  $3113121,~~ 1213113, ~~4220224, ~~3413131,~ ~1313143, ~ ~4224224$.\\ 
  Then $\dim_\R H^2(\OC_X, \R)= 1$.  \vspace{.2cm}
   
   \item Assume the parametrization of the orbit $\OC_X$ is given by any of the sequences :\\
   $2000002,~ 0101010,~ 2002002,~ 1110111,~ 2020020,~ 0200202, ~1112111, ~2022020, ~0202202$, 
   $2202022, ~0220220$.
   Then $\dim_\R H^2(\OC_X, \R) \leq 1$. \vspace{.2cm}
    
   \item Assume the parametrization of the orbit $\OC_X$ is given by any of the sequences :\\
   $2010001, 1000102$, $0120101, 1010210, 1030010$,  $0100301, ~3013010,~ 0103103$.\\
   Then $\dim_\R H^2(\OC_X, \R) \leq 2$. \vspace{.2cm}
      
   \item If the parametrization of the orbit $\OC_X$ is given by either $1010101 $ or $0020200$, then \\
   $\dim_\R H^2(\OC_X, \R) \leq 3$. \vspace{.2cm}
      
   \item   If $\OC_X$ is not given by the parametrizations as in (1), (2), (3), (4), (5), (6) above
 (\# of such orbits are 39), then we have  $\dim_\R H^2(\OC_X, \R)= 0$.
 \end{enumerate}
\end{theorem}

\begin{proof}
For the Lie algebra $E_{7(7)}$, we can easily compute $\dim_\R \z(\k)$ from the last column of \cite[Table XI, pp. 513-514]{D1} 
and $K/K^\circ$ from \cite[Column 4, Table 8, pp. 260-264]{K}.

The orbit $\OC_X$, as given in (1), is of type I as $\dim_\R \z(\k) =3$ and  $K/K^\circ = {\rm Id}$. 
For the orbits, as given in (2), we have $\dim_\R \z(\k) = 2$ and  $K/K^\circ = {\rm Id}$. Hence these are also of type I. 
For the orbits, as given in (3), we have $\dim_\R \z(\k) = 1$ and  $K/K^\circ = {\rm Id}$; hence they are of type I.
For the orbits, as given in (4), we have $\dim_\R \z(\k) = 1$ and  $K/K^\circ = \Z_2$. Thus these are of type II. 
For the orbits, as given in (5), we have $\dim_\R \z(\k) = 2$ and  $K/K^\circ = \Z_2$. Hence  these are also of type II.
For the orbits, as given in (6), we have $\dim_\R \z(\k) = 3$ and  $K/K^\circ \neq {\rm Id}$, hence they are of type II.
Rest of the 39 orbits, which are not given by the parametrizations in (1), (2), (3), (4), (5), (6), are of type III as $\z(\k) = 0$. 
Now the results follow from Corollary \ref{cor-dim-types-ABC}. 
\end{proof}

\subsubsection{\texorpdfstring {Nilpotent orbits in $E_{7(-5)}$.}{Lg}}
There are 37 nonzero nilpotent orbits in $E_{7(-5)}$; see \cite[Table XII, p. 515]{D1}. Note that in this case $\m = [\m,\m]$.

\begin{theorem}\label{E-7(-5)}
Let the parametrization of the nilpotent orbits be as in \S \ref{parametrization-inner-type}. Let $X$ be a nonzero nilpotent element in  $E_{7(-5)}$.
 \begin{enumerate}
  \item If the parametrization of the orbit $\OC_X$ is given by either    
  $110001~1$ or $000120~2$, then $\dim_\R H^2(\OC_X, \R)= 2$. \vspace{.1cm}
  
  \item Assume the parametrization of the orbit $\OC_X$ is given by any of the sequences :\\
  $000010~1,  ~~ 010000~2,~ ~000010~3, ~~ 010010~1,~~ 200100~0,~~010100~2,~~ 000200~0,~~ 010110~1$, 
  $010030~1, \, \ ~ 010110~3, \,  \   201031~4, \, \ 010310~3$.
  Then $\dim_\R H^2(\OC_X, \R)= 1$.\vspace{.2cm}
  
  \item If the parametrization of the orbit $\OC_X$ is given by either $020200~0$ or $111110~1$, then 
  $\dim_\R H^2(\OC_X, \R)\leq 2$.      \vspace{.2cm}
  
  \item Assume the parametrization of the orbit $\OC_X$ is given by any of the sequences :\\
  $020000~0$, $ 201011~2$,  $ 040000~4$, $040400~4$.  
  Then $\dim_\R H^2(\OC_X, \R) \leq 1$.    \vspace{.2cm}

  \item  If $\OC_X$ is not given by the parametrizations as in (1), (2), (3), (4) above
  (\# of such orbits are 17), then we have  $\dim_\R H^2(\OC_X, \R)= 0$.
 \end{enumerate}
\end{theorem}

\begin{proof}
For the Lie algebra $E_{7(-5)}$, we can easily compute $\dim_\R \z(\k)$ from the last column of \cite[Table XII, pp. 515]{D1}
and $K/K^\circ$ from \cite[Column 4, Table 9, pp. 266-268]{K}.

For the orbit $\OC_X$, as in (1), we have $\dim_\R \z(\k) = 2$ and  $K/K^\circ = {\rm Id}$. Hence these orbits are of type I. 
For the orbit $\OC_X$, as in (2), we have $\dim_\R \z(\k) = 1$ and  $K/K^\circ = {\rm Id}$. Hence these orbits are also of type I.
For the orbit $\OC_X$, as in (3), we have $\dim_\R \z(\k) = 2$ and  $K/K^\circ = \Z_2$, hence are of type II.
For the orbit $\OC_X$, as in (4), we have $\dim_\R \z(\k) = 1$ and  $K/K^\circ = \Z_2$. Hence these are also of type II.
Rest of the 17 orbits, which are not given by the parametrizations in (1), (2), (3), (4), are of type III as $\z(\k) = 0$. 
Now the conclusions follow from Corollary \ref{cor-dim-types-ABC}. 
\end{proof}

\subsubsection{\texorpdfstring {Nilpotent orbits in $E_{7(-25)}$.}{Lg}}
There are 22 nonzero nilpotent orbits in $E_{7(-25)}$; see \cite[Table XIII, p. 516]{D1}.  In this case we have $\m \neq [\m,\m]$.

\begin{theorem}\label{E-7(-25)}
Let the parametrization of the nilpotent orbits be as in \S \ref{parametrization-inner-type}. Let $X$ be a nonzero nilpotent element in  $E_{7(-25)}$.
 \begin{enumerate}
 \item  Assume the parametrization of the orbit $\OC_X$ is given by any of the sequences :\\
 $000000\,2,~ 000000 -2,~  000002 -2,~ 200000 -2,~ 200002 -2,~ 400000 -2 ,~   000004 -6$, \\
 $200002 -6, ~  400004 -6, ~ 400004 -10$.
 Then $\dim_\R H^2(\OC_X, \R)= 0$.\vspace{.2cm}
 
 \item If $\OC_X$ is not given by any of the above parametrization (\# of such orbits are $12$),
   then we have  $\dim_\R H^2(\OC_X, \R) \leq 1$.
 \end{enumerate}
\end{theorem}

\begin{proof}
 Note that the parametrization of nilpotent orbits in $E_{7(-25)}$ as in \cite[Table 10]{K} is different from \cite[Table X III, p. 516]{D1}. As the component group for all orbits in $E_{7(-25)}$ is Id; see \cite[Column 4, Table 10, pp. 269-270]{K}, it does not depend on the parametrization. We refer to the last column of \cite[Table X III]{D1} for the orbits as given in (1). These are type III as $\z(\k)=0$. 
 For rest of the 12 orbits we have $\dim_\R \z(\k) = 1$; see last column of \cite[Table X III]{D1}. As $\m \neq [\m,\m]$, these are of type II. Now the results follow from Corollary \ref{cor-dim-types-ABC}. 
\end{proof}

\subsection{\texorpdfstring{Nilpotent orbits in non-compact real forms of $E_8$}{Lg}}

Recall that up to conjugation there are two non-compact real forms of $E_8$. They are denoted by $E_{8(8)}$ and $E_{8(-24)}$.

\subsubsection{\texorpdfstring {Nilpotent orbits in $E_{8(8)}$.}{Lg}}
There are 115 nonzero nilpotent orbits in $E_{8(8)}$; see \cite[Table XIV, pp. 517-519]{D1}. Note that in this case we have $\m=[\m,\m]$.

\begin{theorem}\label{E-8(8)}
Let the parametrization of the nilpotent orbits be as in \S \ref{parametrization-inner-type}. Let $X$ be a nonzero nilpotent element in  $E_{8(8)}$.
 \begin{enumerate}
  \item  Assume the parametrization of the orbit $\OC_X$ is given by any of the sequences :\\
  $10010011,~ 11110010,~ 10111011,~ 11110130$. Then $\dim_\R H^2(\OC_X, \R)= 2$.\vspace{.2cm}
  
  \item  Assume the parametrization of the orbit $\OC_X$ is given by any of the sequences :\\
  $01000010,~ 10001000, ~ 30000001, ~ 10010001, ~ 01010010, ~ 01000110,~  10100100, ~ 00100003$,\\
  $11001030, ~10110100, ~ 21010100, ~ 01020110, ~ 30001030, ~ 11010101, ~ 11101011, ~ 11010111$, \\
  $11111101, ~21031031, ~ 31010211, ~ 12111111, ~ 13111101, ~ 13111141, ~ 13103041, ~ 31131211$,\\
  $13131043, ~34131341$.  
  Then $\dim_\R H^2(\OC_X, \R)=  1$.  \vspace{.2cm}
  
  \item  If the parametrization of the orbit $\OC_X$ is given $00100101$, then $\dim_\R H^2(\OC_X, \R) \leq 3$.  \vspace{.2cm}
  
  \item  Assume the parametrization of the orbit $\OC_X$ is given by any of the sequences :\\
   $10001002,~  10101001, ~ 01200100, ~ 02000200, ~ 10101021, ~ 10102100,~  02020200, ~ 01201031$. \\
   Then $\dim_\R H^2(\OC_X, \R)\leq  2$.  \vspace{.2cm}
   
  \item  Assume the parametrization of the orbit $\OC_X$ is given by any of the sequences :\\
  $11000001,~  20010000, ~ 01000100,~  11001010, ~ 20100011,~  01010100,~  02020000,~  20002000$, \\
  $20100031,~  10101011,~  00200022,~  11110110,~  01011101,~  01003001,~  11101101,~  11101121$,\\
  $10300130,~  04020200,~  02002022,~  00400040,~  11121121,~  30130130,~  02022022,~  40040040$.\\
  Then $\dim_\R H^2(\OC_X, \R)\leq  1$.  \vspace{.2cm}
   
  \item  If $\OC_X$ is not given by the parametrizations as in (1), (2), (3), (4), (5) above
   (\# of such orbits are $52$), then we have  $\dim_\R H^2(\OC_X, \R) =0$.  
 \end{enumerate}
\end{theorem}

\begin{proof}
For the Lie algebra $E_{8(8)}$, we can easily compute $\dim_\R \z(\k)$ from the last column of \cite[Table XIV, pp. 517-519]{D1} and $K/K^\circ$ from \cite[Column 4, Table 11, pp. 271-275]{K}.

For the orbits $\OC_X$, as given in (1), we have $\dim_\R \z(\k) = 2$ and  $K/K^\circ = {\rm Id}$. Hence these orbits are of type I.
For the orbits $\OC_X$, as given in (2), we have $\dim_\R \z(\k) = 1$ and  $K/K^\circ = {\rm Id}$. Hence these orbits are also of type I.
For the orbit  $\OC_X$, as given in (3), we have $\dim_\R \z(\k) = 3$ and  $K/K^\circ \neq {\rm Id}$; hence they are of type II.
For the orbits $\OC_X$, as given in (4), we have $\dim_\R \z(\k) = 2$ and  $K/K^\circ \neq {\rm Id}$. Thus these orbits are of type II.
For the orbits $\OC_X$, as given in (5), we have $\dim_\R \z(\k) = 1$ and  $K/K^\circ \neq {\rm Id}$. Hence these are of type II.
Rest of the 52 orbits, which are not given by the parametrizations of (1), (2), (3), (4), (5), are of type III as $\z(\k) = 0$. 
Now the conclusions follow from Corollary \ref{cor-dim-types-ABC}.
\end{proof}

\subsubsection{\texorpdfstring {Nilpotent orbits in $E_{8(-24)}$.}{Lg}}
There are 36 nonzero nilpotent orbits in $E_{8(-24)}$; see \cite[Table XV, p. 520]{D1}. Note that in this case we have $\m= [\m,\m]$.

\begin{theorem}\label{E-8(-24)}
Let the parametrization of the nilpotent orbits be as in \S \ref{parametrization-inner-type}. Let $X$ be a nonzero nilpotent element in  $E_{8(-24)}$.
 \begin{enumerate}
  \item Assume the parametrization of the orbit $\OC_X$ is given by any of the sequences :\\
  $0000001\,1,  ~~ 1000000\,2,~~   0000001\,3,~~ 1000001\,1 ,~~  1100000\,1,~~ 1000010\,2, ~~ 0000012\,2, ~~ 1000011\,1$,
  $ 1000011\,3, ~~ 1000003\,1,~~   0110001\,2,~~  1010011\,1,~~  1000031\,3$. 
  Then $\dim_\R H^2(\OC_X, \R)= 1$.\vspace{.2cm} 
  
  \item If the parametrization of the orbit $\OC_X$ is given by either $ 2000000\,0$ or $2000020\,0$, then 
  $\dim_\R H^2(\OC_X, \R)\leq 1$.   \vspace{.2cm}
  
  \item  If $\OC_X$ is not given by the parametrizations as in (1), (2) above
   (\# of such orbits are 21), then we have  $\dim_\R H^2(\OC_X, \R)= 0$.  
 \end{enumerate}
\end{theorem}

\begin{proof}
For the Lie algebra $E_{8(-24)}$, we can easily compute $\dim_\R \z(\k)$ from the last column of \cite[Table XV, p. 520]{D1} and $K/K^\circ$ from \cite[Column 4, Table 12, pp. 277-278]{K}.

For the orbits $\OC_X$, as given in (1), we have $\dim_\R \z(\k) = 1$ and  $K/K^\circ = {\rm Id}$, hence these are of type I. 
For the orbits $\OC_X$, as given in (2), we have $\dim_\R \z(\k) = 1$ and  $K/K^\circ \neq {\rm Id}$. Hence these orbits are of type II. 
Rest of the 21 orbits, which are not given by the parametrizations of (1), (2), are of type III as $\z(\k) = 0$. 
Now the conclusions follow from Corollary \ref{cor-dim-types-ABC}. 
\end{proof}

\begin{remark}{\rm 
 Here we make some observations about the first cohomology groups of the nilpotent
orbits in non-compact non-complex real exceptional Lie algebras. To do this we begin
by giving a convenient description of the first cohomology groups of the nilpotent orbits.
Following the set-up of Theorem \ref{thm-nilpotent-orbit} it can be shown that 
\begin{align}\label{H1-exceptional}
 \dim_\R H^1(\OC_X, \R)  =  \begin{cases}
                                        1  & \text{ if }  \,  \k + [\m,\m]\, \subsetneqq \, \m  \\
                                        0  & \text{ if }  \,  \k + [\m,\m]\, =\, \m.
                               \end{cases}
\end{align}
The proof of the above result will appear elsewhere. As a consequences of \eqref{H1-exceptional}, for all the nilpotent orbit $\OC_X$ in a simple Lie algebra $\g$
we have $\dim_\R H^1(\OC_X, \R) \leq 1$.
Recall that  if $\g$ is a non-compact non-complex real exceptional Lie algebra such that $\g \not\simeq E_{6(-14)}$ and  $\g \not\simeq E_{7(-25)}$ then any maximal compact subgroup of ${\rm Int}\, \g$ is semisimple, and
hence, using \eqref{H1-exceptional}, it follows that $ \dim_\R{ H}^1(\OC_X, \R) = 0$ for all nilpotent orbit $\OC_X$ in $\g$. We next assume
$\g = E_{6(-14)}$ or  $\g = E_{7(-25)}$. Note that in both the cases $[\m,\m] \subsetneqq \, \m $.
We follow the parametrizations of the nilpotent orbits of $\g$ as given in \cite[Tables X, XIII]{D1}; see  \S 3.1.1 also. 
When $\g = E_{6(-14)}$ we are able to conclude that $\dim_\R H^1( \OC_X, \R)=1$ only for one orbit, namely, the orbit $\OC_X$ parametrized by $40000\,-2$.
In this case, from the last column and row 9 of \cite[Table X, p. 512]{D1} one has $\k = [\k,\k]$. Thus $\k + [\m,\m]\, \subsetneqq \, \m$, and \eqref{H1-exceptional} applies.  
For $\g = E_{7(-25)}$ we obtain that $\dim_\R H^1( \OC_X, \R)=1$ when $\OC_X$ is parametrized by any of the following sequences :
$000000\,2$,~  $ 000000 -2$,~ $ 000002 -2$,~ $ 200000 -2$,~ $200002 -2$,~ $400000 -2 $,  $ 000004 -6$, ~ $200002 -6$,~  $ 400004 -6$, ~$ 400004 -10$.
For the above orbits, from the last column of \cite[Table XIII, p. 516]{D1} we have $\k= [\k,\k]$, 
and hence, using \eqref{H1-exceptional}, analogous arguments apply.
 }
\end{remark}

\end{document}